\documentclass[final,12pt]{article}
\usepackage[utf8]{inputenc}
\usepackage[T1]{fontenc}

\usepackage{ifthen}
\usepackage{hyperref}

\ifdefined \bDoNotIncludePackages \else
  \newcommand{\bDoNotIncludePackages}{0}
\fi
\ifdefined \bSkipDocumentSetting \else
  \newcommand{\bSkipDocumentSetting}{0}
\fi
\ifdefined \bDoNotDefineTheorems \else
  \newcommand{\bDoNotDefineTheorems}{0}
\fi

\ifthenelse{\equal{\bDoNotIncludePackages}{1}}{}
{
    \usepackage{graphicx}
    \usepackage{amsmath,amsthm,amssymb,mathtools}
    \usepackage{enumerate}
    \usepackage{cleveref} 
    \usepackage{subfig}
    \usepackage{color}
}

\ifthenelse{\equal{\bSkipDocumentSetting}{1}}{}{

\addtolength{\voffset}{-1cm} 
\addtolength{\hoffset}{-1.5cm} 
\setlength{\textheight}{22cm} \setlength{\textwidth}{16cm}
}

\def\N{\mathbb N}
\def\A{\mathcal A}
\def\B{\mathcal B}
\def\C{\mathcal C}

\def\P{\mathcal P}

\def\S{\mathcal S}

\def\mF{\mathcal F}
\def\E{\mathcal E}

\def\P{\mathcal P}

\def\L{\mathcal L}
\def\Lu{{\mathcal L}(\uu)}
\def\Lv{{\mathcal L}(\vv)}
\def\uu{\mathbf u}
\def\vv{\mathbf v}
\def\ww{\mathbf w}

\def\Pal{{\rm Pal}}
\def \id {{\rm Id}}

\def\SS{S}

\def \Rk#1 {$\mathcal{R}_{#1}$}
\def \Rext {{\rm Rext}}
\def \Lext {{\rm Lext}}

\def \Pext {{\rm Pext}}

\def \b {{\rm b}}

\def \FC#1 {
\mathcal{C}
\ifthenelse{\equal{#1}{}}{}{(#1)}
}

\def \PC#1 {
\mathcal{P}_{\Theta}
\ifthenelse{\equal{#1}{}}{}{(#1)}
}

\def \PCn#1 {
\mathcal{P}
\ifthenelse{\equal{#1}{}}{}{(#1)}
}

\ifthenelse{\equal{\bDoNotDefineTheorems}{1}}{}
{
\newtheorem{thm}{Theorem}
\newtheorem{theorem}[thm]{Theorem}

\newtheorem{lemma}[thm]{Lemma}

\newtheorem{proposition}[thm]{Proposition}

\newtheorem{conjecture}[thm]{Conjecture}

\theoremstyle{definition}

\newtheorem{definition}[thm]{Definition}
\newtheorem{example}[thm]{Example}

\crefname{thm}{theorem}{theorems}
\crefname{thrm}{theorem}{theorems}
\crefname{coro}{corollary}{corollaries}
\crefname{example}{example}{examples}
\crefname{lem}{lemma}{lemmas}
\crefname{lmm}{lemma}{lemmas}
\crefname{claim}{claim}{claims}
\crefname{obs}{observation}{observations}
\crefname{proposition}{proposition}{propositions}
\crefname{prop}{proposition}{propositions}
\crefname{defi}{definition}{definitions}

\crefname{theorem}{theorem}{theorems}
\crefname{corollary}{corollary}{corollaries}
\crefname{example}{example}{examples}
\crefname{lemma}{lemma}{lemmas}
\crefname{proposition}{proposition}{propositions}
\crefname{definition}{definition}{definitions}

\theoremstyle{remark}
\newtheorem{remark}[thm]{Remark}

\crefname{example}{example}{examples}
}

\begin{document}


\title{Morphic images of episturmian words \\ having finite palindromic defect}

\author{\v St\v ep\'an Starosta \\ Czech Technical University in Prague \\ Czech Republic}

\date{}
\maketitle

\begin{abstract}
We study morphisms from certain classes and their action on episturmian words.
The first class is $P_{ret}$.
In general, a morphism of class $P_{ret}$ can map an infinite word having zero palindromic defect to a word having infinite palindromic defect.
We show that the image of an episturmian word, which has zero palindromic defect, under a morphism of class $P_{ret}$ has always its palindromic defect finite.
We also focus on letter-to-letter morphisms to binary alphabet: we show that images of ternary episturmian words under such morphisms have zero palindromic defect.
These results contribute to the study of an unsolved question of characterization of morphisms that preserve finite (resp. zero) palindromic defect.
They also enable us to construct new examples of binary $H$-rich and almost $H$-rich words, where $H = \{\id, R, E, RE \}$ is the group generated by both involutory antimorphisms on a binary alphabet.
\end{abstract}






\section{Introduction}

In combinatorics on words, the most famous class of words probably is the class of Sturmian words: aperiodic words having minimal factor complexity possible (see \cite{HeMo}).
Sturmian words are profoundly studied and many generalizations are known, see for instance \cite{BaPeSta2}.
One such generalization of Sturmian words are episturmian words.
Episturmian words were inspired by Arnoux--Rauzy words (see \cite{Ra83,ArRo}).
An infinite word over a $k$-letter alphabet is \textit{episturmian} if it is closed under reversal and has at most one left special factor of each length.
Refer for instance to \cite{DrJuPi,JuPi,GlJu} for more results on this class.

A notion related to the study of episturmian words is a \textit{palindrome} -- a word equal to its reversal.
Episturmian words are rich in palindromes: they contain the maximum number of distinct palindromic factors possible.
Precisely, we say that a finite word $w$ is \textit{rich} if it contains exactly $|w|+1$ distinct palindromic factors, which is the upper bound for the number of distinct palindromic factors in a finite word of length $|w|$ (see \cite{DrJuPi}).
The notion is extended to infinite words: an infinite word is \textit{rich} if every its factor is rich.

In the context of this upper bound on the number of palindromic factors, a measure of the count of missing palindromic factors was introduced in \cite{BrHaNiRe}: the \textit{palindromic defect} $D(w)$ of a finite word $w$ is
$$
D(w) = |w|+1 - \# \Pal(w),
$$
where $\Pal(w)$ is the set of all palindromic factors of $w$.
The palindromic defect of an infinite word $\uu$ is defined by $D(\uu) = \sup \{ D(w) \colon w \text{ is a factor of } \uu \}$.
If $D(\uu)$ is finite, we say that $\uu$ is \textit{almost rich}.
(If it is zero, then $\uu$ is rich as already mentioned.)

Besides episturmian words, examples of rich words include some well-explored word classes such as words coding symmetric interval exchange transformation (see \cite{BaMaPe}) and words coding rotation on two intervals (see \cite{BlBrLaVu11}).
Properties and characterizations of rich words are studied for instance in \cite{GlJuWiZa, BaPeSta2, BuLuGlZa2, BuLuLu}.
Examples of words that are have finite and nonzero palindromic defect are also known.
Such words can be constructed using iterated palindromic closure: let $w_0 \in \A^*$ with $\A$ being an alphabet.
Set $w_i = (w_{i-1}\delta_{i})^R$ where $\delta_i \in \A$ and $w^R$ is the shortest palindrome having $w$ as a prefix (i.e., the \textit{palindromic closure} of $w$).
The infinite word having $w_i$ as its prefix for all $i$ is the \textit{standard word with seed $w_0$ and directive sequence $(\delta_i)_{i=0}^{+\infty}$}.
It follows from \cite{BuLuLuZa2} that such a word is almost rich.
General properties and characterizations of almost rich words are studied in \cite{GlJuWiZa,BaPeSta3,BaPeSta5}.

In this paper, we study richness and almost richness of images of episturmian words by a morphism from a specific class.
Our first result states that we obtain an almost rich word while using a morphism of class $P_{ret}$ introduced in \cite{BaPeSta3} (see \Cref{sec:ClasesOfMorphisms} later for the definition).

\begin{theorem}\label{Th:PretMapsToRich} Let $\uu \in \A\mathbb{^N}$ be an
episturmian word and $\pi: \A^*\to  \B^*$
be a morphism of class $P_{ret}$. The word  $\pi(\uu)$ is almost
rich.
\end{theorem}

The second main result involves a letter-to-letter projection of a ternary episturmian word to binary alphabet.
We use the following definition for a such a projection.

\begin{definition} \label{def:BinarniProjekce}
Let $\A$ be an alphabet and  $\A'$ its proper subset.
A morphism $\zeta: \A \to \{A,B\}$  defined by
$$
\zeta: a \mapsto \begin{cases} A \text{ if } a \in \A', \\ B \text{
otherwise.} \end{cases}
$$
is called a {\it binary projection} from $\A$.
\end{definition}

The second main result states that we obtain a rich word by projecting a ternary episturmian word to a binary alphabet.

\begin{theorem} \label{ARprojekcerich}
Let $\uu$ be an episturmian word over a ternary alphabet $\A$ and
$\zeta$ be a binary projection from $\A$. The word $\zeta(\uu)$ is
rich.
\end{theorem}

Our motivation for these results is to find new binary words which are rich in a generalized sense --- with respect to both symmetries given by the involutive antimorphisms on a binary alphabet: the reversal $R$ and the exchange of letters $E$.
We give a definition in \Cref{sec:Application}, see also \cite{PeSta1,PeSta2} for more information on this generalization.
To construct new binary words rich in this generalized sense, we use the recent results of \cite{PeSta3} which provides theorems that relate the classical richness and the generalized richness on the binary alphabet.

Our computer experiments suggest that we can improve \Cref{ARprojekcerich}: we can drop the requirement on the size of the alphabet $\A$.
We state this hypothesis in the last section.

The paper is organized as follows.
The next section contains some necessary definitions and basic results.
\Cref{sec:MorsAndProofs} contains overview of results on episturmian words and proofs of the main results.
Finally, \Cref{sec:Application} contains an application of the main results: a construction of binary words which are rich and almost rich in the generalized sense.
The last section states some comments and open questions.


\section{Preliminaries} \label{sec:prelim}

\subsection{Notions of combinatorics on words}

Let $\A$ be an \emph{alphabet} --- a finite set of \emph{letters}.
A finite sequence $w = w_0w_1 \cdots w_{n-1}$ with $w_i \in \A$ for all $i$ is a \emph{finite word}.
The \emph{length} of the word $w$ is denoted $|w|$ and equals $n$.
The unique word of length $0$ is the \emph{empty word}, it is denoted $\varepsilon$.
The set $\A^*$ is the set of all finite words over $\A$.
The set $\A^*$ equipped with concatenation forms a free monoid with the neutral element $\varepsilon$.
A word $v \in \A^*$ is a \emph{factor} of a word $w \in \A^*$ if $w= uvz$ for some words $u,z\in \A^*$.
If $u =\varepsilon$, then $v$ is  a \emph{prefix} of $w$; if $z =\varepsilon$, then $v$ is a \emph{suffix} of $w$.
If $w$ is of the form  $w=vz$, then $z$  is denoted  $z=v^{-1}w $  and the word $v^{-1}w v$ is a \emph{conjugate} of the word $w$.

An \emph{infinite word} over $\A$ is a sequence $\uu = (u_j)_{j\in \mathbb{N}} = u_0u_1u_2\ldots$.
The set of all infinite words over $\A$ is denoted by $\A^\N$.
A finite word $w \in \A^*$ of length $n$ is a \emph{factor} of $\uu  = (u_j)_{j\in \mathbb{N}}$ if there exists an integer $i$ such that $w=u_iu_{i+1} \cdots u_{i+n-1}$.
The integer $i$ is an \emph{occurrence} of the factor $w$ in $\uu$.
The \emph{language} of $\uu$ is the set of all its factors and is denoted by $\L(\uu)$.
Given $a \in \A$ and $w \in \A^*$, if $wa \in  \L(\uu)$, then $wa$  is a \emph{right extension} of the factor $w$.
The set of all right extensions of $w$ is denoted $\Rext(w)$.
Any factor of $\uu$ has at least one right extension.
If $w$ has at least two right extensions, it is \emph{right special}.
The notions of \emph{left extension} of a factor, \emph{left special} factor and $\Lext(w)$ are defined analogously.
A factor $w$ which is left and right special is \emph{bispecial}.

An infinite word $\uu$ is \emph{recurrent} if any factor of $\uu$ has infinitely many occurrences in $\uu$.
If for every factor the sequence of all its consecutive occurrences has its first difference bounded, then the word $\uu$ is \emph{uniformly recurrent}.
Let $rw$ be a factor of $\uu$ such that $rw$ has a prefix $w$ and $w$ occurs as a factor in $rw$ exactly twice.
Such a word $r$ is a \emph{return word} of $w$ and the word $rw$ is a \emph{complete return word} of $w$.
An infinite word $\uu$ is uniformly recurrent if and only if every its factor has finitely many return words.

An infinite word $\uu$ is \textit{eventually periodic} if there exist words $p$ and $z$ such that $\uu = pzzz\ldots = pz^\omega$.
It is \textit{periodic} if $p = \varepsilon$.
If an infinite word is not eventually periodic, it is \textit{aperiodic}.

We denote by $\C_{\uu}$ the mapping $\mathbb{N} \to \mathbb{N}$ which is determined by $\C_\uu(n) = \# ( \L(\uu) \cap \A^n)$, i.e., it counts the factors of length $n$ of the word $\uu$.
This mapping is the \emph{factor complexity} of $\uu$.

Given two alphabets $\A$ and $\B$, a mapping $\mu: \A^* \to \B^*$ is a \emph{morphism} if $\mu(wv) = \mu(w)\mu(v)$ for all $w,v \in \A^*$.
It is an \emph{antimorphism} if $\mu(wv) = \mu(v)\mu(w)$ for all $w,v \in \A^*$.
An infinite word $\uu$ is \emph{closed under the mapping $\mu$} if  for any factor $w \in \L(\uu)$ we have also $\mu(w) \in \L(\uu)$.
A morphism $ \nu: \A^* \to \B^*$ is a \emph{conjugate morphism} to a morphism $\mu$ if there exists a word  $w \in \B^*$ such that for any letter $a \in \A$ we have $w\mu(a) = \nu(a)w$.

An antimorphism $\Psi$ is \emph{involutory} if $\Psi^2=\id$.
The most frequent involutory antimorphism is the reversal mapping $R$ which is given by
$$
R(w_0w_1 \cdots w_{n-1}) = w_{n-1} \cdots w_0 \quad \text{with }w_i \in \A.
$$
It can be easily seen that if a word $\uu$ is closed under an involutory antimorphism, then  $\uu$ is recurrent.

If $w = R(w)$, then $w$ is a \textit{palindrome}.
The set of all palindromes occurring as factors of a finite word $w$ is denoted $\Pal(w)$.
The \emph{palindromic complexity} of an infinite word $\uu$ is the mapping $P_{\uu}: \N \to \N$ defined by
$\P_\uu(n) = \# \{p \in \L(\uu) \colon p= R(p), |p| = n \}$, i.e., the number of palindromic factors of length $n$.

A palindrome $w$ is \emph{centered at $x \in \A \cup \{\varepsilon\}$} if $w = vxR(v)$ for some word $v$.
Obviously, a palindrome is centered at $\varepsilon$ if and only if it is of even length.

Let $w \in \A^*$.
The quantity measuring the number of missing palindromic factors in $w$ is the \textit{(palindromic) defect} of $w$, introduced in \cite{BrHaNiRe} and  given by
$$
D(w) = |w| + 1 - \# \Pal(w).
$$
The notion of defect is extended to infinite words in the following way: let $\uu$ be an infinite word, we set
$$
D(\uu) = \sup \{ D(w) \colon w \in  \Lu \}.
$$
If the defect of a finite or infinite word is zero, the word is fully saturated by distinct palindromic factors and is said to be \textit{rich} (or sometimes \textit{full}).
If its defect is finite, it is \textit{almost rich}.

To prove (almost) richness of a word we will use the characterization of rich words given in \cite{BaPeSta}.
It exploits the notion of bilateral order $\b(w)$  of a factor $w$ and palindromic extension of a~palindrome.
The \emph{bilateral order} was introduced in \cite{Ca} as
\begin{equation}\label{eq:bilateralorder}
\b(w)   =  \# \{ awb \in \Lu \colon a,b \in \A \} - \# \Rext(w) - \#\Lext(w)+1.
\end{equation}
If $a$ and $b$ are letters, we say that $awb$ is a \textit{both-sided extension of $w$} if $awb \in \Lu$.
Moreover, if $w$ is a palindrome and $a = b$, then $awa$ is its \textit{palindromic extension}.
The set of all palindromic extensions of a palindrome $w \in \Lu$ is denoted by $\Pext(w)$: we have
$$ \Pext(w) =\{ awa \colon awa \in \mathcal{L}({\mathbf u}), a \in \mathcal{A} \}  .$$

\begin{theorem}[\cite{BaPeSta2}] \label{thm:minus1}
Let $\uu$ be an infinite word that is closed under
reversal.
The word $\uu$ is rich  if and only if any bispecial factor
$w$ of $\uu$ satisfies
\begin{equation}\label{eq:BilateralProRich}
\b(w) = \begin{cases} \# \Pext(w) - 1 & \text{ if $w$ is a
palindrome;} \\ 0 & \text{ otherwise.} \end{cases}
\end{equation}
\end{theorem}

The next theorem is another characterization of rich words which will be useful.

\begin{theorem}[\cite{BuLuGlZa}] \label{thm:rich_crw_gen}
A word $\uu$ is rich if and only if for every $w \in \Lu$, any factor of $\uu$ containing exactly two occurrences of $w$ or $R(w)$, one as a prefix and one as a suffix, is a palindrome.
\end{theorem}

\subsection{Words with infinitely many palindromes and morphisms} \label{sec:ClasesOfMorphisms}

Hof, Knill and Simon studied  in \cite{HoKnSi}  words having infinitely many palindromes.
They considered words that are constructed using morphisms of the following class:

\begin{definition}\label{classP}     A primitive morphism $\varphi: \A^* \to \A^*$   belongs to  the class $P$
 if there exists a palindrome $r\in \A^*$  such that for any letter $a \in \A$ we have that $r$ is a prefix of $\varphi(a)$ and   $r^{-1}\varphi(a)$ is a
palindrome.
\end{definition}

The class $P$ is related to words with infinitely many palindromes
in the following sense. Let $\varphi$ be a morphism of class $P$. If
$v$ is a palindrome, then $\varphi(v)r$ is also a palindrome. Thus,
a fixed point of a morphism of class $P$ has infinitely many
palindromes.

The question is the converse: if one has a uniformly recurrent fixed point with infinitely many palindromes, can it be generated by a morphism of class $P$?
A discussion on the meaning of this formulation can be found in \cite{LaPe14}.
In \cite{BoTan}, the author gave the following affirmative answer to this question in case of a binary alphabet:  if a  binary  fixed point of a primitive morphism  $\varphi$  contains infinitely  many palindromes, then  $\varphi$ or $\varphi^2$ is conjugate to a morphism from class $P$.
In \cite{La2013}, it is demonstrated that this claim cannot be generalized to a larger alphabet even under the assumption of injectivity of the morphism.
In \cite{LaPe14}, the authors show that if an infinite word is a fixed point of a marked primitive morphism and contains infinitely many palindromes, then some power of the morphism has a conjugate in class P.

We now mention two classes of morphisms that are closely related to the class $P$.
The first class is defined in~\cite{GlJuWiZa}.
A~morphism $\varphi$ is a~{\em standard morphism of class $P$} (or a~{\em standard $P$-morphism}) if there exists a~palindrome $r$ (possibly empty) such that for all $x \in \A, \ \varphi(x) = rq_x$ and $q_x$ is a palindrome.
If $r$ is non-empty, then some (or all) of the palindromes $q_x$ may be empty or may even take the form $q_x = \pi_x^{-1}$ with $\pi_x$ a~proper palindromic suffix of $r$.
A~standard morphism of class $P$ is {\em special} if
\begin{enumerate}
\item whenever $\varphi(x)r = rq_xr$, with $x \in {\mathcal A}$, occurs in some $\varphi(y_1y_2\cdots y_n)r$, then this occurrence is $\varphi(y_m)r$ for some $m$ with $1 \leq m \leq n$; and
\item all $\varphi(x) = rq_x$ end with different letters.
\end{enumerate}

Let $\varphi$ be a special standard morphism of class $P$.
It is related to richness in the following way: if $\uu$ is rich, then $\varphi(\uu)$ is almost rich, see \cite{GlJuWiZa}.

The second class of morphisms related to the class P is the class $P_{ret}$.
In \cite{BaPeSta3}, this class was used to show relations between rich and almost rich words in general: every almost rich word is an image of a rich word by a morphism of this class.
The class was also used in \cite{BuLu} to show that every episturmian word is an image of an Arnoux--Rauzy word.
In \cite{HaVeZa}, it is used to show an answer to an interpretation of the previously mentioned question of Hof, Knill and Simon.

\begin{definition} \label{def:Pret}
Let $\varphi : \mathcal{B}^*\to  \mathcal{A}^*$ and $r \in \A^*$ be a palindrome.
We say that $\varphi$ is of class $P_{ret}$ (with respect to $r$) if the following is true:
\begin{itemize}
\item \label{def_Pret_1}  $\varphi(b)r$ is a palindrome for any $b\in \mathcal{B}$,
\item \label{def_Pret_2} $\varphi(b)r$ contains exactly $2$ distinct occurrences of $r$, one as a prefix and one as a suffix, for any
$b\in \mathcal{B}$,
 \item \label{def_Pret_3} $\varphi(b)\neq \varphi(c)$ for all $b, c \in \mathcal{B},\
b\neq c$.
\end{itemize}
\end{definition}
A direct consequence of  the above definition is that any morphism
 $\varphi \in P_{ret}$  is injective and $\varphi(s)r$ is a palindrome if and only if $s\in \mathcal{B}^*$ is a palindrome.
An important property of the class is also that it is closed under taking composition of morphisms, see \cite{BaPeSta3}.

The following example taken from \cite{BaPeSta3}  illustrates that there exists a morphism of class $P_{ret}$ which maps a word having finite defect to a word with infinite defect.
One of the main results of this article, \Cref{Th:PretMapsToRich}, states that is cannot happen if such a morphism acts on an episturmian word.

\begin{example} \label{ex:Pret}
Let $v_0 = \varepsilon$ and for all $i > 0$ set
$$
v_{i} = v_{i-1}0v_{i-1}1v_{i-1}1v_{i-1}0v_{i-1}2v_{i-1}2v_{i-1}0v_{i-1}1v_{i-1}1v_{i-1}0v_{i-1}.
$$
Let $\vv \in \{0,1,2\}^\N$ be determined by the limit
$$
\vv = \lim_{i \to + \infty} v_i
$$
and $\varphi: \{0,1,2\}^* \to \{0,1\}^*$ be given as follows:
$$
\varphi: \left \{ \begin{array}{l} 0 \mapsto 0100 \\ 1 \mapsto 01011 \\ 1 \mapsto 010111 \end{array} \right..
$$
Proposition 5.7 in \cite{BaPeSta3} states that $D(\vv) = 0$ and $D(\varphi(\vv)) = +\infty$.
In other words, the morphism $\varphi$, which is of class $P_{ret}$, maps a rich word to a word which is not almost rich.

The key property of $\vv$ in this case is that the palindromes $v_i$ have two palindromic extensions $1v_i1$ and $2v_i2$ which produce (by application of $\varphi$) the same palindrome $1 \varphi(v_i) 010 1$ but they also produce a non-palindromic complete return word to it (it is contained in $\varphi(1v_{i}1v_{i}0v_{i}2v_{i}2)$).
Having infinitely many non-palindromic complete return words to palindromes implies that the defect is infinite (see \cite{GlJuWiZa}).
\end{example}


\section{Morphic images  of  episturmian words versus richness}

\label{sec:MorsAndProofs}

In this section, we first recall and deduce some properties of episturmian words that will be needed later.
A proof of  Theorem  \ref{Th:PretMapsToRich} is given in the second subsection.
The last subsection contains a proof of \Cref{ARprojekcerich}.

\subsection{Properties of episturmian words}

For basic reference on episturmian words, the reader can refer to \cite{DrJuPi, JuPi} or survey papers \cite{GlJu, Be_survey}.

As already mentioned, an infinite word $\uu \in \A^{\N}$ is \emph{episturmian} if for
any $n$ there exists at most one left special factor of length $n$
and $\uu$ is closed under reversal. If  for any $n$ there exists
exactly one left special factor of length $n$ with $\#\A$ left
extensions, then $\uu$ is a $k$-ary \emph{Arnoux--Rauzy} word with
$k = \# \A$.

The most  important examples of episturmian words can be constructed using the palindromic closure.
The standard word with the directive sequence $(\delta_i)_{i=0}^{+\infty} \in \A^\N$ and with the seed $w_0 = \varepsilon$, i.e., the word $\uu$ such that
$$
\uu = \lim_{n \to +\infty} w_n \quad \quad \text{where  } w_n = (w_{n-1}\delta_{n-1})^R, 
$$
 is an episturmian word over $\A$ and it is called \emph{standard episturmian}.
The importance follows from the fact that given any episturmian word $\uu$, there exists a unique standard episturmian word with the same language.
Since richness is a property of language, when studying it, we can restrict ourselves to the standard episturmian words which are characterized by their directive sequence.
A standard episturmian word can be recognized by looking at its prefixes: an episturmian word is standard if and only if each its prefix is a left special factor.

A basic property of any episturmian word $\uu$ is that one letter is
\emph{separating}: it occurs in every factor of length $2$. The
separating letter  of $\uu$ is the first letter of the directive
sequence $\Delta = \delta_1\delta_2\delta_3 \ldots$ of the corresponding standard
episturmian word.
Denote  $\ell$ the least integer such that $\delta_1^\ell$
is not a prefix of $\Delta$.  Then the word $x\delta_1^ky$ with $x,y \neq
\delta_1$ is a factor of the episturmian word $\uu$ only if $k=\ell-1$ or
$k=\ell$.  Moreover, the word $\delta_1^\ell$ is a factor of the
episturmian word $\uu$  if and only if the letter $\delta_1$ occurs at
least once in the sequence
$\delta_{\ell+1}\delta_{\ell+2}\delta_{\ell+3}\ldots$.

Any palindromic prefix of a standard episturmian word is equal to $w_n$ for some $n \in \N$.
The complete return words of $w_n$ are described by Theorem 4.4 of \cite{JuVu00}.
Denote by $\mF$ the set $\{ \delta_m \colon m \geq n \}$.
Every complete return word of $w_n$ equals to the palindromic closure
\begin{equation*}\label{CRWofBS}
(w_nx)^R= q_xw_n
\end{equation*}
where $x \in \mF$.
The word $q_x$ is the corresponding return word.
Thus, the palindromic prefix $w_n$ has exactly $\# \mF$ return words.

Bispecial factors play a crucial role in the study of the language of an infinite word.
We give essential properties of bispecial factors of an episturmian word.
Any bispecial factor $w$  of an episturmian word $\uu$ is a palindrome; moreover there exists $n \in \N$ such that $w = w_n$, where $w_n$ is a palindromic prefix of the corresponding standard episturmian word.
Denoting again $\mF = \{ \delta_m \colon m \geq n \}$, the set of all non-palindromic both-sided extensions of the bispecial factor $w$ satisfies
\begin{equation}\label{eq:extensions}
\{xwy \in \L(\uu) \colon x,y \in \A, x \neq y \} =  \{ awx \in \L(\uu) \colon  x \in \mF \setminus \{ a \} \} \cup \{ xwa \in \L(\uu) \colon  x \in \mF  \setminus \{ a \} \},
\end{equation}
where $a \in \mF$ is the unique letter such that $awx \in \Lu$ and $xwa \in \Lu$ for some letter $x \in \A$.
If $\uu$ is Arnoux--Rauzy, then $awa \in \Lu$ (see \cite{DaZa03}) (and $\mF = \A$); if $\uu$ is not Arnoux--Rauzy, then it depends on $w$ whether $awa \in \Lu$ (and there exist a bispecial factor $w$ such that $awa \not \in \Lu$).

The next lemmas give an insight on the structure of return words
of an episturmian word.

\begin{lemma} \label{AR_rw_projekce}
Let $\uu$ be an episturmian word over $\A$ and $w \in \Lu$ be a
factor. Let $ r_1, r_2,\ldots, r_{s}$ be the list of all distinct
return words of $w$ in $\uu$. Define a morphism $\Psi$ over
$\E = \{1,2,\ldots,s\}$  by the rule $ k \mapsto
r_k$ for all $k \in  \E$.
There exists an episturmian word $\vv$  over $\E$ such that $\uu = g\Psi(\vv)$ for some finite word $g$.
\end{lemma}

\begin{proof}
First, let us  assume that $w$ is a bispecial factor of $\uu$.
 If $\uu$ is  an  Arnoux--Rauzy word, then $\# \E = \# \A$ and  the claim follows directly from the
proof of Theorem 1 in \cite{BuLu} and from Theorem 3 ibidem.
In case
$\# \E < \# \A$,  the proof is analogous and is
left to the reader.

Suppose $w$ is not bispecial and it can be extended in a unique way to the shortest bispecial factor  $b=uwv$ in which it occurs.
A factor  $r$ is a return word of $b$ if and only if  $u^{-1}ru$ is a
return word of $w$.
Thus the morphism $\Psi$ (defined using the return words of $w$) is a
conjugate morphism of the morphism defined using the return words of $b$ and we can use the
validity of the statement for the bispecial factor $b$.

The last case, $w$ is not bispecial and it cannot be extended to a bispecial factor, is trivial as  $w$ has only one return word and thus  $\# \E = 1$.
\end{proof}

The word $\vv$ from \Cref{AR_rw_projekce} captures the structure of return words of the factor $w$ in $\uu$.
Such a word was also studied in \cite{Durand98} and we keep the same terminology:
we say that the word $\vv$ from the previous lemma is a \textit{derivated word of $\uu$  with respect to the factor $w$} and its corresponding morphism is $\Psi$.
By comparing the definitions we obtain that if $w$ is a palindrome, then $\Psi \in P_{ret}$.

\begin{lemma} \label{AR:delky}
Let $\uu$ be  a standard episturmian word $\uu$ over the  alphabet $\A=\{0,1,2\}$ with the
directive sequence $\delta_1\delta_2\delta_3\ldots$, where
$\delta_1=0$. Denote by $\ell$ the least integer such that
$0^\ell$ is not a prefix of the directive sequence.

\begin{enumerate}
\item If  $w =1$ or $w=2$, then no two return words of $w$ have the same length.
\item If  there exists $j$ such that $j >\ell$ and $\delta_j=0$, then $w=0^{\ell}$ is a factor of $\uu$ and no two return words of $w$
have the same length.
\end{enumerate}
\end{lemma}

\begin{proof}
If $w$ has exactly $1$ return word, then the claim is trivially satisfied.

Suppose $w$ has more than $1$ return word.
Let $b=uwv$ be the shortest bispecial factor containing  $w$.
 Any return word of $w$ has the form $u^{-1}ru$, where $r$ is a return
   word of $b$.   Therefore, it is enough to prove the statement only for bispecial
   factors.   Since $\uu$ has a prefix $0^{\ell -1}$ but not $0^\ell$, the bispecial  factor $b$ contains at least two distinct letters.

As already mentioned, Theorem 4.4 of \cite{JuVu00} describes the return words of bispecial factors.
In particular, any complete return word of $b$ equals to the palindromic closure of $bx$ for some letter $x\in \A$.
Moreover, since $\uu$ is standard, each bispecial factor $b$ of $\uu$  is a palindromic prefix of $\uu$.
Since the prefixes of the word are constructed using the palindromic closure and the directive sequence, the palindromic closures $(bx)^R$ and $(by)^R$ for distinct letters $x$ and $y$ have the same length if and only if neither $x$ nor $y$ occurs in $b$ (see \cite{Ju05}).
As $b$ contains at least two distinct letters, this cannot happen on the ternary alphabet.
\end{proof}

\begin{remark} \label{delsirw}
As one can see from the proof of the last lemma, for alphabets having cardinality greater than $3$ the claim of the last lemma does not hold. Let us suppose that $\uu$ is a standard episturmian word with its directive sequence $\Delta
=01023\ldots$. We have
$$
\uu = 010010201001030100102010010 \ldots
$$
Clearly,   $\ell =2$. The factors  $0010201$ and  $0010301$ are return words of $w=0^\ell = 00$  and they are of  the same length.
\end{remark}

\begin{lemma} \label{lem:epist_navraty_k_pismenum}
Let $\uu$ be an episturmian word over $\A$. Let $\A' \subset \A$
and $xpy$  be a factor of $\uu$ such that  $x,y \in \A'$ and $p$
does not contain any letter from $\A'$. The  word $p$ is a
palindrome.
\end{lemma}
\begin{proof}
We can suppose without loss of generality that the word $\uu$ is standard episturmian.
We order the elements of $\A'$ according to their first occurrence in the directive sequence, i.e., $$\A' = \{ x_1, x_2, \ldots, x_n \},$$
where $x_i$ appears for the first time in the directive sequence of $\uu$ before the first occurrence of $x_j$ if and only if $i < j$.

Suppose $x = x_1$.
Let $p_1$ be the longest prefix of $\uu$ that does not contain $x_1$ (or equivalently, the longest prefix of $\uu$ over $\A \setminus \A'$).
Since $\uu$ is a standard word constructed using the palindromic closure operators, the word $p_1$ is a palindrome (it may be empty).
The shortest palindromic prefix containing $x_1$ is $p_1x_1p_1$.
Using Theorem 4.4 of \cite{JuVu00} we have that any complete return word of $p_1x_1p_1$ has the form  $(p_1x_1p_1z)^R$ for  some $z \in \A$.
If $p_1x_1p_1$ is bispecial, it is the shortest bispecial containing $x_1$; if it is not bispecial, then $\uu$ is periodic.
In both cases, any word $x_1py$ is a prefix of a complete return word of $p_1x_1p_1$ without the leading $p_1$.
If $z \in \A'$, then $y = z$ and $xpy = x_1p_1z$, i.e., $p = p_1$ is a palindrome.
If $z \not \in \A'$, then $y = x_1$ and $xpy = p_1^{-1}(p_1 x_1p_1z)^Rp_1^{-1}$, i.e., $p$ is again a palindrome.

Take $j > 1$ and suppose $x = x_j$.
Let $p_j$ be the longest prefix of $\uu$ that does not contain $x_j$.
The shortest palindromic prefix of $\uu$ containing $x_j$ is the word $p_jx_jp_j$.
If $p_jx_jp_j$ is bispecial, it is the shortest bispecial factor containing $x_j$; if it is not bispecial, then $\uu$ is periodic.
Both cases imply that the any complete return word of $x_j$ starts with $x_jp_j$.
Since $x_jp_j = x_jp_1x_1s$ for some word $s$, there is just one word of the desired form, precisely $x_jp_1x_1$, i.e., $y=x_1$, and the proof is finished.
\end{proof}

When dealing with images by a morphism a useful notion is an ancestor.
Suppose that $w$ is a factor of $\Psi(\vv)$ for some infinite word $\vv$ and morphism $\Psi$.
A factor $e_0e_1\cdots e_n \in \Lv$ is an{ \it ancestor} by $\Psi$ of $w$  if  $\Psi(e_0e_1\cdots e_n)$ contains $w$  but neither $\Psi(e_1\cdots e_n)$ nor
$\Psi(e_0e_1\cdots e_{n-1})$ contains $w$.

The following lemma can be considered as a generalization of \Cref{thm:rich_crw_gen}.
The lemma uses an additional notion.
Let us fix a  factor $w$   of an episturmian word $\uu$ over the alphabet $\A$ and a subset $\E \subset \A$.
A both-sided extension  $xwy$ of $w$ in $\uu$  is  called  {\em $\E$-extension} if $x$ and $y$  belong  to
$\E$.

\begin{lemma}\label{lem:zobecneneReturn}   Let $w$ be a palindromic factor of an episturmian word
$ \uu\in \A^\N$ and $a \in \mathcal{A}$  be the unique
letter satisfying  $awx, xwa \in\Lu$ for some letter $x$. Let $\E\subset \A$ such that $a \in \E$.
 If $xuy$ is a factor of $\uu$ containing exactly two
$\E$-extensions of $w$ --- one as its prefix and one as its
suffix --- then $u$ is a palindrome.
\end{lemma}

\begin{proof}

If  $w$ is not a bispecial factor,  then  the only both-sided
extension of $w$ is $awa$. It means that $xuy=aua$ is a complete  return word of the palindrome $awa$ and thus, using \Cref{thm:rich_crw_gen}, the word $u$ is a
palindrome.

For a bispecial factor $w$,  we denote by $n$ the number of occurrences of the  factor  $w$ in $u$.

If $n =2$, then $u$ is a complete return word of $w$ and, since $\uu$ is rich, using \Cref{thm:rich_crw_gen} $u$ is a palindrome.

If $n >2$, we first show that the factor $xuy$ begins with $xwa$ and ends with $awy$.
Suppose that $xuy$ begins with $xwy$, $y \neq a$.
There is only one complete return word of $w$ beginning with $wy$, it ends with $yw$ and it is always followed by the letter $a$, thus we have a contradiction to $n > 2$.

We deal with two subcases:

1) \quad  $x=y$. \\
Since the word  $xuy$ has a  prefix $p=xwa$, it has a suffix $R(p)=awx$.
The factor $xuy=xux$ does not contain further occurrences of $p$ and $R(p)$.
According to \Cref{thm:rich_crw_gen}, such a factor is a palindrome in any rich word, in particular in the episturmian word $\uu$, as we want to show.

2) \quad  $x\neq y$. \\
Let $\B$ be the set of letters such that for all $z \in \B$ the word $wz$ is a right extension of $w$.
For all $z \in \B$, denote $r_z$ the return word of $w$ which has a prefix $wz$.
Define a morphism $\Psi: \B^* \to \A^*$ by setting $\Psi(z) = r_z$ for all $z \in \B$.
Using \Cref{AR_rw_projekce}, let $\vv \in \B^\N$ be the episturmian word which is a derivated word of $\uu$ with respect to the factor $w$ such that $\Psi$ is the corresponding morphism.
The ancestor by $\Psi$ of the factor $xuy$ is a factor $xasay \in \L(\vv)$ with $s$ containing only letters  from $\{ z : z=a \text{ \  or \ } z \notin
\E \}$ and moreover, in any pair of neighboring letters in $asa$ exactly one of the letters is $a$, i.e., $aa$ does not occur in $asa$.

 If $x,y \neq a$, then we apply Lemma \ref{lem:epist_navraty_k_pismenum} with $\A' = \E \setminus \{a\}$.
We get that $asa$ is a palindrome and thus $u=\Psi(asa)w$ is a palindrome as well.

If $x=a$ and $y\neq a$, then since $a$ is the  separating letter in $\vv$, we have that $xasay = aasay$ is an image of the morphism $\sigma_a$ defined by $a  \mapsto a$ and  $b\mapsto ab $   for any $b\neq a$.
Thus  $aasay = \sigma_a(a \tilde{s}y)$ where  $a \tilde{s}y$ is a factor of an episturmian word and $\tilde{s}$ is produced from $s$ by erasing all letters $a$ in it.
Moreover, the word $\tilde{s}$ does not  contain  letters from  $\E$.   Applying  Lemma \ref{lem:epist_navraty_k_pismenum}  with   $\A' = \E$ we get that  $  \tilde{s}$ is a palindrome. Since $\sigma_a$ and $\Psi$ belong to the class $P_{ret}$, their composition $\Psi\sigma_a$ is in $P_{ret}$  with respect to $\Psi(a)w$. 
Consequently,  $u =     \Psi(\sigma_a( \tilde{s})a)w$ is a palindrome.
\end{proof}

\begin{remark}
The morphism $\sigma_a$ is one of \emph{episturmian morphisms}, see \cite{JuPi}.
For instance, episturmian morphisms can be used to construct episturmian words.
\end{remark}


\subsection{Proof of Theorem  \ref{Th:PretMapsToRich}}

Besides the proof of Theorem \ref{Th:PretMapsToRich} which
concerns images  of episturmian words under a morphism of class
$P_{ret}$, we also provide  in this section a description of the
bilateral order of long bispecial factors in these infinite words.
This result is  needed also to prove \Cref{ARprojekcerich} in the next section.

\begin{lemma} \label{lem:epis_Pret}
Let $\uu$ be an episturmian word and let $\pi: \A^* \to \B^*$ a morphism
of class $P_{ret}$ with respect to the palindrome $r$. If $w \in
\L(\pi(\uu))$ is a bispecial factor containing at least one
occurrence of $r$, then its bilateral order $b(w)$ is $-1$ or $0$,  and satisfies
\eqref{eq:BilateralProRich}.

\end{lemma}

\begin{proof}
Let $w \in \L(\pi(\uu))$ be a bispecial factor containing $r$. Let
us denote $sr$ and $rp$ the prefix and the suffix of $w$ containing
 the factor $r$ exactly once. As $\pi$ belongs to the class
$P_{ret}$, there exists a unique factor $v \in \Lu$ such that $w = s
\pi(v) r p$.


Let $\{z_1, \ldots, z_t\}$ be all the letters such that $wz_i$ is a right extension of $w$.
We denote by $R_i = \{ b \in \A \colon \pi(v)rp z_i \text{ is a prefix of } \pi(vb)r \text{ and } vb \in \Lu \}$.
This set is well defined due to the choice of $p$ and $v$.
Since $\pi$ is of class $P_{ret}$, it follows that $R_i \cap R_j = \emptyset$ if $i \neq j$.
Thus, $v$ is right special.

Let $\{y_1, \ldots, y_\ell\}$ be all the letters such that $y_iw$ is a left extension of $w$.
Analogously to $R_i$, we define the sets $L_i$.
We conclude that $v$ is left special, thus it is bispecial.

Denote by $a$ the unique letter such that $avx \in \Lu$ and $xva \in \Lu$ for some $x \in \A$.
Next, we show that there exists an index $i$ such that $a \in R_i$.
Suppose the contrary: for all $i$ and $b \in R_i$ we have that $wz_i$ is a factor of $\pi(cvb)$ for some letter $c$ (possibly depending on $b$).
Using \eqref{eq:extensions}, since $b \neq a$, it implies that $c = a$, i.e., $v$ is not special --- a contradiction.
Let without loss of generality $a \in R_1$.
Analogously, we have $a \in L_1$.

Using the property \eqref{eq:extensions} for $v$, we have that if $y_jwz_k \in \L(\pi(\uu))$, then $k > 1$  implies $j = 1$ and symmetrically $j > 1$ implies $k = 1$.
Thus we have
$$
\{ ywz \in \L(\pi(\uu)) \} \setminus \{y_1wz_1\} = \{y_1wz_k \colon k > 1 \} \cup \{ y_kwz_1 \colon k > 1 \}.
$$
To calculate the bilateral order of $w$, it remains to see if $y_1wz_1$ is its both-sided extension or not:
if $y_1wz_1 \in \L(\pi(\uu))$, then $b(w) = 0$; otherwise $b(w) = -1$.

We first show that if $R_1 = L_1 = \{a\}$, then $w$ is a palindrome.
Since $R_1 = \{a\}$, there is no letter $b \in R_i$ for some $i$ such that the longest common prefix of $\pi(b)r$ and $\pi(a)r$ is longer than $|rp|$.
Similarly, there is no letter $c \in L_j$ for some $j$ such that the longest common suffix of $\pi(c)$ and $\pi(a)$ is longer than $|s|$.
This implies $\bigcup R_i = \bigcup L_i$, and $s = R(rp)$, i.e., $w$ is a palindrome.
Moreover, we obtain $t = \ell$, and for all $i$ there exists $j$ such that $R_i = L_j$.
It implies that $y_iwz_j$ is a palindrome if and only if $i = j = 1$.

We have that $y_1wz_1 \in \L(\pi(\uu))$ if and only if $ava \in \Lu$ or $R_1 \neq \{a\}$ or $L_1 \neq \{a\}$.

Thus, if $y_1wz_1 \not \in \L(\pi(\uu))$, i.e., $b(w)  = -1$, then  $R_1 = L_1 = \{a\}$, which implies that $w$ is a palindrome and $w$ has no palindromic extension.
Therefore, $b(w) = \# \Pext(w) -1 = -1$ and \eqref{eq:BilateralProRich} is satisfied.

The last case is $y_1wz_1 \in \L(\pi(\uu))$.
If $w$ is a palindrome, since $y_1wz_1$ is the only palindromic extension, \eqref{eq:BilateralProRich} is again satisfied.
If $w$ is not a palindrome, we have $b(w) = 0$, and \eqref{eq:BilateralProRich} is also satisfied.
\end{proof}

\begin{proof}[Proof of  Theorem \ref{Th:PretMapsToRich}]
Let $\pi$ be a morphism from $P_{ret}$ with respect to $r$.
Since $\uu$ is a derivated word of $\pi(\uu)$ with respect to $r$, then $\uu$ is periodic if and only if $\pi(\uu)$ is periodic.

Suppose $\uu$ is periodic.
Thus, $\pi(\uu)$ is also periodic.
In Theorem 4 of \cite{BrHaNiRe}, it is shown that an infinite periodic word contains infinitely many palindromes if and only if it is equal to $(pq)^\omega$ for some palindromes $p$ and $q$.
Since $\uu$ contains infinitely many palindromes, $\pi(\uu)$ contains infinitely many palindromes and we can write $\pi(\uu) = (pq)^\omega$ for some palindromes $p$ and $q$ from $\L(\pi(\uu))$.
It can be easily seen that such a word has finite defect.

Suppose $\uu$ is aperiodic.
As shown in \cite{BaPeSta3}, a uniformly recurrent infinite word $\pi(\uu)$ containing infinitely many palindromes has finite defect  if and only if there exists an integer $M$ such that any complete return word of any palindrome $w$ of length at least $M$ is palindromic.
As for each palindrome $p \in \Lu$, the word $\pi(p)r$ is a palindrome, and moreover a factor of $\pi(\uu)$, it remains to verify palindromicity of  complete return words to palindromic factors of length $N$ such that $N \geq M =  2 \max \{ |\pi(a)r| \colon a \in \A \}$.
Let $w$ be a palindromic factor of $\pi(\uu)$ such that $|w| \geq M$.

First, suppose that $w$ is a bispecial factor.
The word $w$ contains $r$ and  according to the proof of  Lemma \ref{lem:epis_Pret} there exists a bispecial palindromic factor $v$, a
letter $a$  and a factor $s$ such that for some letter $x$ we have  $$w = s\pi(v)rR(s), \ avx \in \Lu, \ xva \in \Lu \ \text{ and } \ sr \text{ is a suffix of } \pi(a)r.$$
 Let $f$ be a complete return word of $w$.  Clearly, $f= s\pi(u)rR(s)$ for some factor $u \in \uu$
and $v$ is a prefix and a suffix of $u$.  If $s$ is  the empty word
$\varepsilon$, then  $u$ is a complete return word of $v$ in $\uu$.
Since $u$ is a palindrome, the factor $f=\pi(u)r$ is a palindrome as
well.

 If $s\neq \varepsilon$, then we apply Lemma \ref{lem:zobecneneReturn} with $\E = \{b \in \A
\colon sr  \text{ is a suffix of }  \pi(b)r\}$.
According to our notation any ancestor of $w$ has the form $xvy$,
where $x,y \in \E$. Moreover, any ancestor of the complete return
word $f= s\pi(u)rR(s)$ equals  $xuy$ with $x,y \in \E$ and the
factor $xuy$ has only two occurrences of $\E$-extensions of the factor $v$.
According to Lemma \ref{lem:zobecneneReturn}, the factor  $u$ is a
palindrome and thus $\pi(u)r$ is a palindrome as well. It implies that
the complete return word $f = s\pi(u)rR(s)$  is a palindrome, too.

To complete the proof, we need to discuss the case when $w$ is not bispecial.
Since $\pi(\uu)$ is closed under reversal, such a palindrome $w$ has a unique both-sided extension and this extension is palindromic.
It implies  that  the shortest bispecial factor in which the palindrome  $w$ occurs is a bispecial palindrome, say $qwR(q)$.
Moreover, $u$ is a complete return word of $w$ if and only if $quR(q)$ is a complete return word of $u$.
Since $quR(q)$ is palindrome, $u$ is palindrome as well.
\end{proof}

\begin{remark}
Given an episturmian word $\uu$ and a morphism $\pi$ of class $P_{ret}$, the defect of $\pi(\uu)$ can be nonzero.
It suffices to choose $\pi$ such that for some letter $a \in \A$ the defect of $\pi(a)$ is nonzero, which is possible in general.

For instance, let $\uu$ be the Fibonacci word, i.e., the fixed point of the morphism given by $0 \mapsto 01$ and $1 \mapsto 0$.
The Fibonacci word is Sturmian, i.e., an episturmian word.
Let $\pi \in \P_{ret}$ be determined by
$$
0 \mapsto 110100110010 \quad \text{ and } \quad 1 \mapsto 1.
$$
(The morphism $\pi$ is of class $P_{ret}$ with respect to the palindrome $r = 11$.)
We have $D(\pi(0)) = 1$, thus $D(\pi(\uu)) \geq 1$.
In fact, $D(\pi(\uu)) = 2$.
\end{remark}


\subsection{Proof of \Cref{ARprojekcerich}}

Before proving \Cref{ARprojekcerich}, we need two lemmas concerned with images of episturmian words by morphisms of a special form.

\begin{lemma} \label{lemma:ABdom}
Let $\uu$ be an episturmian word over $\A$ and $\pi: \A \to \{A,B\}$ a
morphism of class $P_{ret}$ such that $\pi(x) = AB^{q_x}$ for all
$x$. The word $\pi(\uu)$ is rich.
\end{lemma}

\begin{proof}
Any morphism from class $P_{ret}$ is injective, thus $q_x \neq
q_y$ for any $x, y \in \A, x\neq y$.  We will show that if $w$ is a
bispecial factor of $\pi(\uu)$, then its bilateral order $b(w)$ satisfies
\eqref{eq:BilateralProRich}.

If the bispecial factor $w$ contains the letter $A$, the claim
follows from \Cref{lem:epis_Pret}. Suppose $w$ does not contain any
occurrence of the letter $A$, i.e., $w = B^k$ with
$k<\max\limits_{x\in\A} q_x $.

The set of both-sided extensions of $w$ clearly contains $AB^{k+1}$
and $B^{k+1}A$. Thus the set  of both-sided extensions of $w$ has
cardinality $2+ \#\Pext(w)$. Since $\#\Rext(w)=2=\#\Lext(w)$, according
to equation \eqref{eq:bilateralorder} we have $\b(w) = \# \Pext(w) - 1$.

Since \eqref{eq:BilateralProRich} is satisfied for any bispecial factor of $\pi(\uu)$, the richness of $\pi(\uu)$ follows from \Cref{thm:minus1}.
\end{proof}

\begin{lemma} \label{lemma:ABnedom}
Let $\uu$ be a ternary episturmian word over $\{0,1,2\}$ with the separating letter $0$ and
let $\pi: \A \to \{A,B\}$ be a morphism such that $\pi(0) = A$ and
$\pi(x) = AB$ for all $x \neq 0$. The word $\pi(\uu)$ is rich.
\end{lemma}

\begin{proof}
Let $\ell$ be the minimal integer such $0^\ell$  is not a prefix of the directive sequence of the standard episturmian word having the same language as $\uu$.
It means that any two non-zero letters of $\uu$ are separated by the block $0^\ell$ or $0^{\ell-1}$.

If  $0^\ell$ does not belong to $\Lu$, then $\pi(\uu)= (A^{\ell}B)^\omega$, which is a rich word.

Let us suppose that  $0^\ell$ occurs in $\uu$.
Using \Cref{AR_rw_projekce}, let $\vv \in \B\mathbb{^N}$ be an episturmian word that is a derivated word of $\uu$ with respect to the factor $0^\ell$ and let $\Psi$ be the corresponding morphism.
Thus, for any letter $b \in \B$ the image $\Psi(b)$ is a return word of $0^\ell$ in $\uu$.
Using \Cref{AR:delky} one can see that the morphism $\pi \Psi$ is of class $P_{ret}$ with respect to the palindrome $r = A^{\ell+1}$.

We will show that if $w$ is a bispecial factor of $\pi(\uu) $, then
the bilateral order  $b(w)$ satisfies \eqref{eq:BilateralProRich}.
If the bispecial factor $w$ contains $r$, the claim follows from
\Cref{lem:epis_Pret}.

Suppose $w$ does not contain any occurrence of the word $r$. Since
any occurrence of the letter $B$ in $\pi(\uu) $ is followed by block
$A^\ell$ or $A^{\ell+1}$ one can see that the longest factor (not
necessarily bispecial) of $\pi(\Psi(\vv))$ which does not contain $r
= A^{\ell+1}$ is of the form $A^\ell(BA^\ell)^k$ for some integer
$k$.  Thus the  bispecial factor $w$   is of the form
$w=A^\ell(BA^\ell)^s$, with $s<k$.
Since $A^\ell(BA^\ell)^k$ is a factor, we can see that $AwB$ and $BwA$ are factors of $\pi(\uu)$.
Since $w$ is a palindrome, we have $b(w) = \#\Pext(w)-1$
and consequently  $w$ satisfies \eqref{eq:BilateralProRich} as well.

Since \eqref{eq:BilateralProRich} is satisfied for any bispecial factor of $\pi(\uu)$, \Cref{thm:minus1} implies that $\pi(\uu)$ is rich.
\end{proof}

\begin{proof}[Proof of \Cref{ARprojekcerich}]
Without loss of generality let us suppose that $\A = \{0,1,2\}$ and $0$ is the separating letter of $\uu$. 
Let $x \in \A$ be the letter such that $\zeta(x) \neq \zeta(y)$ for all $y \in \A$ such that $y \neq x$.

Suppose $x = 0$.
Let $\vv$  be a derivated word of $\uu$ with respect to the factor $0$ and let $\Psi$ be the corresponding morphism.
Since the return words of $x = 0$ are $0$ and $0y$ for all $y \neq 0$, we may choose $\Psi$ such that $\Psi(0) = 0$.
It implies that the morphism $\pi = \zeta\Psi$ and the word $\vv$ satisfy the assumptions of \Cref{lemma:ABnedom} which implies that $\pi(\vv) = \zeta(\uu)$ is rich.

Suppose $x \neq 0$, i.e., $x$ is not the separating letter in $\uu$.
Let again $\vv$ be a derivated word of $\uu$ with respect to the word $x$ and let $\Psi$ be the corresponding morphism.
Using \Cref{AR:delky}, one can see that $\pi = \zeta\Psi $ is of the form as in the assumptions  \Cref{lemma:ABdom} which implies that $\pi(\vv) = \zeta(\uu)$ is rich.
\end{proof}


\section{Application: construction of $H$-rich words} \label{sec:Application}

A generalization of the notion of richness was introduced in \cite{PeSta1}.
As already mentioned, \Cref{Th:PretMapsToRich,ARprojekcerich} and results of \cite{PeSta3} may be used to produce examples of words that are rich in this generalized sense.
In order to give these examples, and to give the definition of generalized richness on binary alphabets, we need a few more definitions.

Let $E$ be the antimorphism over $\{0, 1\}$ which exchanges the two letters, i.e., for instance we have $E(011) = E(1)E(1)E(0) = 001$.
Let $\Psi \in \{R,E\}$.
If $p = \Psi(p)$, the word $p$ is a \emph{$\Psi$-palindrome}.
The \emph{$\Psi$-palindromic complexity} of an infinite word $\uu$ is
the mapping $\P^\Psi_{\uu}: \mathbb{N} \to \mathbb{N}$, defined by
$\P^\Psi_\uu(n) = \# \{p \in \L(\uu)\colon p= \Psi(p), |p| = n\}$.

Let $H = \{ R, E, RE, \id \}$.
For binary words closed under all elements of $H$, it is shown in  \cite{PeSta1}  that
\begin{equation}
\C_{\uu}(n\!\!+\!\!1) - \C_{\uu}(n) + 4\geq \P^R_{\uu}(n\!\! +\!\!1) + \P^R_{\uu}(n) + \P_{\uu}^E(n\!\!+\!\! 1) + \P_{\uu}^E(n) \quad \text{ for every } n \geq 1.
\end{equation}
When equality is reached for every $n \geq 1$, we say that the word is \textit{$H$-rich}.
When the equality holds except for finitely many $n$, then the word is \textit{almost $H$-rich}.

In \cite{PeSta1}, the notion of generalized richness is defined for arbitrary alphabet and a finite group of morphisms and antimorphisms $G$ containing at least one antimorphism.
In this context, the ``classical'' richness defined in the previous sections is for $G= \{\id, R\}$, and it coincides with (almost) $\{\id, R\}$-richness, or shortly (almost) $R$-richness.

The following construction of $H$-rich and almost $H$-rich words uses the operation $\SS$ acting over words over the alphabet $\{0,1\}$ and defined by   $S(u_0u_1u_2\ldots) = v_1v_2v_3\ldots$, where  $v_i= u_{i-1} + u_i  \mod 2$.
It also uses the two following results of \cite{PeSta3}.

\begin{proposition}[\cite{PeSta3}] \label{centered}
Let  $\vv = \SS(\ww) \in \{0,1\}^\N$ be uniformly recurrent.
If $\Lv$ contains infinitely many $R$-palindromes centered at the letter $1$ and infinitely many $R$-palindromes not centered at the letter $1$,
then $\ww$ is closed under all elements of $H$.
\end{proposition}

\begin{theorem}[\cite{PeSta3}] \label{thm:bin1}
Let $\ww \in \{0,1\}^\N$ and let  $\ww$ be  closed
under all elements of $H = \{ \id, E, R, ER\}$.  Then $\ww$ is
$H$-rich (resp. almost $H$-rich) if and only if $\SS(\ww)$ is
$R$-rich (resp. almost $R$-rich).
\end{theorem}

In order to produce an almost $H$-rich word using the last two claims, we need to find a suitable uniformly recurrent binary almost $R$-rich word $\vv$ which  contains infinitely many $R$-palindromes centered at the letter $1$ and infinitely many $R$-palindromes not centered at the letter $1$.
We will apply a well chosen morphism $\pi \in P_{ret}$ on an episturmian word to produce such a word $\vv$.

Let us recall that  any  Arnoux--Rauzy  word over $\A$ has infinitely many palindromes centered at  $a$ for each $a \in \A \cup\{\varepsilon\}$, see \cite{DaZa03}.
Consider  $\pi: \A^* \to \B^*$   a morphism from the class $P_{ret}$ with respect to a palindrome $r$.
\begin{itemize}
\item If $v$ is a palindrome of even length, the palindrome $\pi(v)r$ is centered at the same letter as $r$.
\item If $v$  is a palindrome of odd length, centered at $a\in \A$, then $\pi(v)r$ is centered at the same letter as $\pi(a)r$.
\end{itemize}
Thus, when we have an Arnoux--Rauzy word $\uu$ over $\A$, we can always pick a morphism $\pi: \A^* \to \{0,1\}^*$ such that $\vv = \pi(\uu)$ has infinitely many palindromes centered at $1$ and infinitely many palindromes not centered at $1$.
Using \Cref{Th:PretMapsToRich}, we have that $\vv$ is almost $R$-rich.
Since $\uu$ is uniformly recurrent, if $\ww$ is a word such that $\vv = \S(\ww)$, then it is closed under all elements of $H$ by \Cref{centered}.
\Cref{thm:bin1} implies that $\ww$ is almost $H$-rich.

Using \Cref{ARprojekcerich}, we can produce $H$-rich words.

\begin{proposition} \label{HrichP}
Let $\uu$ be a ternary Arnoux--Rauzy word over the alphabet $\A=\{0,1,2\}$ and
$\zeta$ be a binary projection over $\A$. 
If   $\ww \in \{0,1\}^\N$ is a preimage of $\zeta(\uu)$ by $S$, i.e.,
$S(\ww) = \zeta(\uu)$, then  $\ww$ is $H$-rich.
\end{proposition}
\begin{proof}  We use  the two well-known properties of Arnoux--Rauzy words which were already mentioned:
any Arnoux--Rauzy word is uniformly recurrent and any Arnoux--Rauzy word contains infinitely many palindromes centered at $a$, where $a \in \A \cup \{\varepsilon\}$.
 Therefore,  $\zeta(\uu)$ is uniformly recurrent and  contains infinitely many palindromes centered at $1$ and infinitely many palindromes centered at $\varepsilon$. Due to  \Cref{centered}, the word $\ww$ is closed under all elements of $H$.  According to \Cref{thm:bin1}, $\ww$ is $H$-rich.
\end{proof}

\section{Comments and open questions}

As already mentioned in the introduction, our computer experiments suggest that we may conjecture a more general statement than   \Cref{ARprojekcerich}:

\begin{conjecture}
Let $\uu$ be an episturmian word over $\A$ and $\zeta$ be a binary projection from $\A$.
The word $\zeta(\uu)$ is rich.
\end{conjecture}

Besides this conjecture, one may inquire about the class of binary words that is obtained by applying a binary projection on an episturmian word.
Study of special factors, factor complexity and structure of return words may provide insight into this class.
It would also be of interest to see if there is a relation to some known class of binary words.


\section*{Acknowledgments}

The author acknowledges financial support from the Czech Science Foundation grant GA\v CR 13-35273P.
The computer experiments were performed using the open-source algebra system \texttt{SAGE} \cite{sage_6}.

\bibliographystyle{siam}
\IfFileExists{biblio.bib}{\bibliography{biblio}}{\bibliography{../!bibliography/biblio}}

\end{document}